\documentclass[12pt]{article}

\usepackage[mathscr]{eucal}
\usepackage{color}
\usepackage{amsmath}
\usepackage{amssymb}
\usepackage{amsfonts}
\usepackage{amscd}
\usepackage{amsthm}
\usepackage{latexsym}
\usepackage{graphicx}
\usepackage{enumitem}

\def\be{\begin{equation}}
\def\ee{\end{equation}}
\def\bse{\begin{subequations}}
\def\ese{\end{subequations}}

\newtheorem{thm}{Theorem}
\newtheorem{cor}[thm]{Corollary}
\newtheorem{lem}[thm]{Lemma}

\newtheorem{rem}[thm]{Remark}

\def\Xint#1{\mathchoice
{\XXint\displaystyle\textstyle{#1}}%
{\XXint\textstyle\scriptstyle{#1}}%
{\XXint\scriptstyle
\scriptscriptstyle{#1}}%
{\XXint\scriptscriptstyle
\scriptscriptstyle{#1}}%
\!\int}
\def\XXint#1#2#3{{
\setbox0=\hbox{$#1{#2#3}{\int}$}
\vcenter{\hbox{$#2#3$}}\kern-.5\wd0}}
\def\dashint{\Xint-}

\def\aint{\dashint}
\def\bse{\begin{subequations}}
\def\ese{\end{subequations}}

\def\bse{\begin{subequations}}
\def\ese{\end{subequations}}

\usepackage{geometry}

\title{Large-time asymptotics of a \\ public goods game model with diffusion}

\author{Klemens Fellner\footnote{Institut f\"ur Mathematik und
    Wissenschaftliches Rechnen, Universit\"at Graz, Heinrichstra{\ss}e 36, 8010
    Graz. Email: \texttt{klemens.fellner@uni-graz.at}. Partially supported by NAWI Graz.}
  \and Evangelos Latos\footnote{Lehrstuhl f\"ur Mathematik IV, Universit\"at Mannheim, D-68131 Mannheim. Email: \texttt{evangelos.latos@math.uni-mannheim.de}.}
  \and Takashi Suzuki\footnote{Graduate School of Engineering Science /
Department of Systems Innovation / Division of Mathematical Science, Osaka University. Email: \texttt{suzuki@sigmath.es.osaka-u.ac.jp}}
}

\date{\today}

\begin{document}

\maketitle

\begin{abstract}
We consider a spatially inhomogeneous public goods game model with diffusion. By utilising a generalised Hamiltonian structure of the model we study the existence of global classical solutions as well as the large time behaviour: First, the asymptotic convergence of the PDE to the corresponding ODE system is proven. This result entails also the periodic behaviour of PDE solutions in the large time limit.
Secondly, a shadow system approximation is considered and the convergence of the PDE  to the shadow system 
in the associated fast-diffusion limit is shown. Finally, the asymptotic convergence of the shadow to the ODE system is proven.
\end{abstract}

\vskip5mm

{\small

}
\vskip5mm

\section{Introduction}

In this work, we are interested in a PDE version of an optional public good game \cite{HMHS} 
\begin{align}
\label{HamPDEsys0}
\begin{cases}
\partial_t f - d_f \Delta f  =   -f(1-f)\,G(z),\quad &x\in\Omega, t>0,
\\[2mm]
\partial_t z - d_z \Delta z =  (\sigma - f (r-1)) \,z(1-z)(1-z^{N-1}),\quad &x\in\Omega, t>0,
\\[2mm]
\frac{\partial}{\partial\nu}(f,z) =0,\quad &x\in\partial\Omega,
\\[2mm]
f(x,0)=f_0(x),\ z(x,0)=z_0(x),\quad &x\in\Omega,
\end{cases}
\end{align}
where $f$ and $z$ are relative fractions of populations 
and we assume 
\begin{equation}\label{id}
0\le f_0(x) \le 1 \qquad \text{and} \qquad 0\le z_0(x) \le 1,
\qquad \forall x \in \Omega.
\end{equation}
Here, $\Omega$ is a bounded domain of $\mathbb{R}^d$ with smooth boundary and outer unit normal $\nu$.
Moreover, $d_f,d_z>0$ are positive diffusion coefficients.

For the remaining parameters, we assume 
\begin{equation}\label{para}
0<\sigma<r-1, \qquad 2<r<N,
\end{equation}
and the function $G(z)$ is given by
\begin{equation}\label{G-function}
	G(z):=1+(r-1)z^{N-1}-\frac{r}{N}\frac{1-z^N}{1-z}.
\end{equation}
Note that in the parameter range \eqref{para}, the function $G(z)$ has exactly one sign change in $z\in(0,1)$ and looks 
qualitatively like Figure~\ref{PlotG}, see Section~\ref{sec:pre} for the details. 
Moreover, $N\ge3$ in \eqref{para} and \eqref{G-function}  denotes the number of players, see Section~\ref{sec:pre} for more details on the considered public good game \cite{HMHS}.
\begin{figure}[htp]\label{PlotG}
\begin{center}
\includegraphics[scale=0.5]{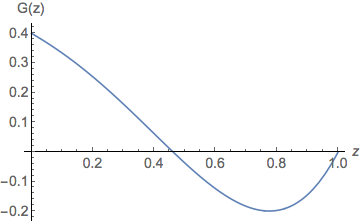}
\caption{A prototypical plot of $G(z)$ with $r=3$ and $N=5$.} 
\end{center}
\end{figure}

Adding diffusion in these kind of models has already been considered in \cite{AGN,Frey,MK}, mostly to model microbial interactions mediated by diffusible molecules, since standard game theory cannot describe such behaviour. Moreover, in contrast to human behaviour or animal colonies, microbial communities rarely rely on direct contact since microbes primarily communicate though diffusible molecules. This diffusive behaviour is the reason why such molecules are often termed public goods, see \cite{AGN,Frey,MK,TH} and the references therein.

This work focuses on the study of the dynamics of PDE-problems of the type \eqref{HamPDEsys0}. 
Herein, the considered optional public good game \cite{HMHS} should be viewed as an interesting example case that leads us to study the  
general question of links between PDE and ODE model.
In fact, we expect our mathematical analysis to similarly apply to related models to \eqref{HamPDEsys0},
which shares the below considered key properties.
While global existence of classical solutions of \eqref{HamPDEsys0} is straightforward, we are in particular interested in the asymptotic large-time behaviour of the solutions and their qualitative properties. 

More precisely, a main question of this paper asks if 
PDE model \eqref{HamPDEsys0}, despite having 
sign changing terms at the right hand side of both equations, exhibits the same large-time behaviour as the corresponding ODE-model, which was originally studied in \cite{HMHS}.
\medskip

The key structural property, which will allow to 
characterise the large-time behaviour of the PDE model 
\eqref{HamPDEsys0} is that the original ODE-model \cite{HMHS} features in the parameter range \eqref{para} a generalised Hamiltonian structure of the form 
\begin{equation}\label{HamsysOri}
\begin{cases}
\dot{f} = \frac{\partial H_2}{\partial z}\,\phi(f,z),\\[2mm]
\dot{z} = -\frac{\partial H_1}{\partial f}\,\phi(f,z),
\end{cases}
\qquad \text{where}\qquad \phi(f,z)=f(1-f)z(1-z)(1-z^{N-1})
\end{equation}
with a Hamiltonian 
\begin{align}
H (f,z):= H_1(f) + H_2(z),\qquad \text{and, thus}\qquad
\frac{d}{dt} H(f(t),z(t)) =0,
\label{Hamiltonian}
\end{align}
where $H_1$ and $H_2$ are defined below in \eqref{H1} and \eqref{H2}.


The first theorem shows that PDE solutions become spatially homogeneous  as $t\uparrow+\infty$ subject to \eqref{HamsysOri}. The proof requires the technical Lemma  
\ref{HzzPosSemDef}, which provides sufficient conditions  to the positive definiteness of the (Hessian of the) Hamiltonian $H(f,z)$.

\begin{thm}[Global existence and convergence to the ODE] \label{thm:1}\hfill\\
Given $f_0, z_0\in C^2(\overline{\Omega})$ with finite Hamiltonian $H(f_0,z_0)<+\infty$ and $\partial\Omega\in C^2$. 
Assume  
the Hessian of the Hamiltonian $H(f,z)$ to be positive definite, which holds, for instance, under the assumptions of Lemma \ref{HzzPosSemDef}.

Then, a unique global-in-time classical solution to \eqref{HamPDEsys0} exists. Moreover, given a PDE solution $(f(\cdot,t), z(\cdot,t))$ to system \eqref{HamPDEsys0}, there exists an ODE orbit ${\cal O}=\{ (\tilde f(t), \tilde z(t))\}_{t\geq 0}$, where $(\tilde f, \tilde z)=(\tilde f(t), \tilde z(t))$ is a solution to \eqref{HamsysOri}, with 
\begin{equation} 
\lim_{t\uparrow+\infty}\mbox{dist}_{C^2}((f(\cdot,t), z(\cdot,t)), {\cal O})=0.
 \label{eqn:convergence}
\end{equation} 
Here, 
$\mbox{dist}_{C^2} ((f,z), {\cal O})=\inf_{(\tilde f, \tilde z)\in {\cal O}}
\Vert (f,z)-(\tilde f, \tilde z)\Vert_{C^2}$.  

\end{thm}
\begin{rem}
A global existence result could  also be proved by the method of invariant sets (see e.g. \cite{SmollerBook}) since $0<f_0,z_0<1$ implies $0<f,z<1$ for all times. Yet, by using the Hamiltonian structure of the system, we can show  \eqref{eqn:convergence} and get more information about the global dynamics of system \eqref{HamPDEsys0} as stated by the following results. Moreover, the Hamiltonian approach 
can be extended to systems without invariant sets.
	
\end{rem}
Motivated by \cite{KarSuzYam}, we notice that any solution to the ODE model \eqref{HamsysOri} is periodic and that any PDE orbit  is absorbed into one of the periodic ODE orbits ${\cal O}=\{ (\tilde f(t), \tilde z(t))\}_{t\geq 0}$. Thus, we derive the following consequence of Theorem \ref{thm:1}:
\begin{cor}[Periodicity of the large-time behaviour]\label{Cor}\hfill\\
Let the ODE orbit ${\cal O}=\{ (\tilde f(t), \tilde z(t))\}_{t\geq 0}$ as defined in Theorem \ref{thm:1} be composed of more than one point, i.e. be a non-trivial orbit. Then, the 
PDE solution $(f(\cdot,t), z(\cdot,t))$ is periodic 
in the large-time limit and 
there exists a "phase shift" $\lambda > 0$ such that
	\begin{equation}\label{l}
	\lim_{t\uparrow\infty}\|(f(\cdot,t+\lambda),z(\cdot,t+\lambda))-(\tilde f(t),\tilde z(t))\|_{C^2(\Omega)}
	=0.
	\end{equation}

\end{cor}

Next, we consider the shadow system (see e.g. \cite{nf87}) corresponding to system \eqref{HamPDEsys0}, which is (formally) obtained in the limit $d_z\uparrow+\infty$:
\begin{equation}\label{ssHamPDEsys0}
\begin{cases}
\partial_t F - d_F \Delta F =  -F(1-F)G(Z),\\[2mm]
\frac{d}{dt}Z = Z(1-Z)(1-Z^{N-1}) \aint_\Omega(\sigma - F (r-1)) dx,
\end{cases}
\end{equation}
where now $Z=Z(t)$. Shadow system~\eqref{ssHamPDEsys0} has homogeneous Neumann boundary conditions for $F$ and
considered subject to the initial data 
\begin{equation}
\left. F\right\vert_{t=0}=F_0=f_0(x) \quad \mbox{in $\Omega$}, \qquad \left. Z\right\vert_{t=0}=\overline{z}_0=\int_\Omega z_0 \,dx.  
 \label{eqn:shadow2}
\end{equation} 
Shadow systems are used to approximate the parabolic problem by the equilibrium problem obtained in the limit $d_z\uparrow+\infty$, see e.g \cite{kee78}. 
Accordingly, $F=F(x,t)$ is a space and time dependent function while $Z=Z(t)$ depends only on $t$. The following 
theorem justifies rigorous the shadow system approximation scheme. 
{Note that we can equally consider and prove the following results for the 
shadow system obtained in the limit $d_f\uparrow+\infty$.} 

\begin{thm}[Convergence to the shadow system]\hfill\\
Suppose the assumptions of Theorem~\ref{thm:1} and assume $u_0, v_0\in W^{3,s}(\Omega)$, $s>d$ with smooth boundary $\partial\Omega$. 
Let $(f,z)$ and $(F,Z)$ be the solutions to (\ref{HamPDEsys0})  and (\ref{ssHamPDEsys0}), respectively.  Then, for any $T>0$, holds 
\begin{equation} 
\lim_{d_z\uparrow +\infty}\sup_{t\in [0,T]}\left\{ \Vert f(\cdot, t)-F(\cdot, t)\Vert_{C^2} +\Vert z(\cdot,t)-Z(t)\Vert_{C^2}\right\}=0.  
 \label{eqn:conv}
\end{equation} 
\label{thm:2}
\end{thm} 
\begin{rem}
Note the additional initial regularity $u_0, v_0\in W^{3,s}(\Omega)$ constitutes the minimal regularity, which is required to prove Theorem~\ref{thm:2}. However, standard parabolic regularity implies arbitrarily regularity of solutions for arbitrarily smooth boundaries $\partial\Omega$.  Parabolic smoothing also implies that 
the initial regularity could be relaxed to $u_0, v_0\in L^{\infty}(\Omega)$ as in \eqref{id} if statement \eqref{eqn:conv} 
is relaxed to $t\in[\tau,T]$ for $\tau>0$.
\end{rem}

	The asymptotics of the shadow system are also given by the ODE system as $t\uparrow\infty$. 
	
\begin{thm}[Convergence from the shadow to the ODE system]\label{th3}\hfill\\
Under the assumptions of Theorem~\ref{thm:1}, let $(F,Z)=(F(x,t), Z(t))$ be the solution to the shadow system \eqref{ssHamPDEsys0}. 
Let $(\hat{f}, \hat{z})=(\hat f(t), \hat z(t))$ be the solution to the ODE system \eqref{HamsysOri} subject to initial data
\begin{equation} 
\hat f_0=\hat{f}(0)=\overline{F}_0, \qquad \hat z_0=\hat{z}(0)=\overline{z}_0, 
 \label{eqn:oi}
\end{equation} 
where $\displaystyle{\overline{F}_0=\aint_\Omega f_0(x)dx}$.  Then, it holds that 
\begin{equation} 
\lim_{t\uparrow+\infty}\Vert F(\cdot,t)-\hat{f}(t)\Vert_{C^2}=0
\qquad
\text{and}
\qquad
\lim_{t\uparrow+\infty}\left( Z(t)-\hat{z}(t)\right)=0.
 \label{eqn:conv3}
\end{equation} 
 \label{thm:3} 
\end{thm}

\noindent{\bf Discussion of the results:}
Our paper deals with the asymptotic behaviour of solutions to the PDE model \eqref{HamPDEsys0}. We prove global existence of classical solutions to \eqref{HamPDEsys0} and convergence to the corresponding ODE model \eqref{HamsysOri}. Next, we derive the  interesting Corollary \ref{Cor}, which implies that if we have two PDE solutions of \eqref{HamPDEsys0} that start at different times, asymptotically they will be close in the $C^2$-norm
modulo a suitable phase shift.  

After that, we prove that solutions to \eqref{HamPDEsys0} converge to those of the corresponding shadow system obtained in the limit $d_z\uparrow\infty$ and that solutions of the shadow system converge to those of the corresponding ODE orbit. These results can also be seen as follows: If we start from the PDE system \eqref{HamPDEsys0} but with different initial data, we will get two different solutions that have nevertheless two properties in common. First, they are both attracted from the corresponding shadow system and either they will pass close by or through the solutions of this shadow system. Second is the fact that even though we started with different initial data, asymptotically we will have convergence to the corresponding ODE in both cases as $t\uparrow\infty$ but these two limits may be different, i.e. there could be a phase shift between them. This behaviour is similar to the Lotka-Volterra systems which was noticed in \cite{LSY}.

\noindent{\bf Outline:}
This paper is organised as follows: In Section 2, we recall the modelling background and establish some basic properties of system \eqref{HamPDEsys0}. Theorem \ref{thm:1}, Corollary \ref{Cor}, Theorem \ref{thm:2}, Theorem \ref{thm:3}  are proven in Sections 3, 4, 5, 6, respectively.

\section{Preliminaries: Modelling and Formal Properties}
\label{sec:pre}

Public goods games are generalisations of the prisoner's dilemma to an arbitrary number of players, see e.g. \cite{HS} and the reference therein. 
In the model presented in \cite{HMHS}, $N$ players are chosen randomly from a large population. Every round, these players may either contribute an amount $c$ or nothing at all to a common pool. $\eta_c$ denotes the number of the players who cooperate and $N-\eta_c$ is the number of the players that defect. At every round the common pool is 
increased by an interest rate $r$ and then used to pay back to the players.  
The payoffs for cooperators $P_c$ and defectors $P_d$ are given by 
\[
P_c=-c+rc\frac{\eta_c}{N}, \qquad P_d=rc\frac{\eta_c}{N}, \qquad \text{where} \quad 1<r<N
\]
for the model to be a public goods game, \cite{HS}. However, in this game 
it turns out that defecting is the dominating strategy. 

Hence, the authors of \cite{HMHS} proposed an extended model allowing players to decide whether to participate or not. Those who are unwilling to do so are called "loners" and they will receive a fixed payoff $P_l=\sigma c$ with $0<\sigma<r-1$. The payoff $P_l$ ensures that an entirely cooperating group will profit more than loners while loners will profit than a group solely formed of defectors. The model of \cite{HMHS} thus considers 
three types of persons: the loners (refusing to join the group), the cooperators (who join and contribute) and the defectors (who just join). These groups correspond to payoffs $P_l$, $P_c$, $P_d$ and the relative frequencies of these strategies shall be denoted by $x,y,z$ and satisfy condition that $x+y+z=1$. More precisely, it was derived in  \cite{HMHS} that  
\begin{align*}
P_l&=\sigma,\\
P_d&=\sigma z^{N-1}+r\frac{x}{1-z}\left(1-\frac{1-z^N}{N(1-z)}\right),\\
P_c&=P_d-\underbrace{1+(r-1)z^{N-1}-\frac{r}{N}\frac{1-z^N}{1-z}}_{=:G(z)},
\end{align*}
The sign of $P_d-P_c$, i.e. the sign of the function $G(z)$ plays a key factor in determining whether or not it is better to switch strategy, that is to change from deflection to cooperation or vice versa.

It is straightforward to check that the function $G$ can also written as a polynomial with real coefficients:
\begin{equation}\label{G}
G(z) = 1 + (r-1)z^{N-1} -\frac{r}{N}\frac{1-z^N}{1-z}
= (1-\frac{r}{N}) -\frac{r}{N} \sum_{j=1}^{N-2} z^j + (r-1-\frac{r}{N}) z^{N-1}.
\end{equation} 
Note that for $2<r<N$ those coefficients change sign exactly twice and that Descartes' rule of signs implies that $G(z)$ has either two or zero positive roots. 
In fact, Lemma~\ref{lem:G} below shows that $\lim_{z\to1-}G(z)=0$ from negative values. Hence, since clearly 
$G(0)>0$, the function $G(z)$ undergoes exactly one sign change on $z\in(0,1)$ as in Figure \ref{PlotG} above.

Note that it can be easily verified that when $r\leq2$ then $G(z)$  does not have any root in $(0,1)$ and $G(z)=0\Rightarrow z=1$, which means that defecting is 
the dominate strategy.

By using the constraint $x+y+z=1$, the average payoff
can be written as,
\begin{align*}
\overline{P}&=xP_c+yP_d+zP_l	=\sigma-(1-z^{N-1})\left(
(1-z)\sigma-(r-1)x
\right).
\end{align*}
By introducing $f=\frac{x}{x+y}$ as a new variable
and considering the replicator dynamics $\dot{z}=z(\sigma-\overline{P})$, the authors of \cite{HMHS} obtained the 
following system to be considered for $(f,z)\in[0,1]^2$:
\begin{equation}\label{orig}
\begin{cases}
\dot{f} = -f(1-f)\,G(z),\\
\dot{z} = (\sigma - f (r-1)) \,z(1-z)(1-z^{N-1}).
\end{cases}
\end{equation}
Note that like $G(z)$ changes the sign once, also the factor $(\sigma - f (r-1))$ changes 
its sign once according to the value of $f\in(0,1)$ and \eqref{para}.


For the ODE-system \eqref{orig}, the authors  proved in \cite{HMHS} that for $r\leq2$ there are no fixed points for the system in $(0,1)^2$, while when $r>2$ and $0<\sigma<r-1$ then there exists a unique fixed point in the interior of $(0,1)^2$, which is stable and surrounded by closed orbits. Moreover, they proved that all interior orbits are closed.

In fact, system \eqref{orig} can be written as the generalised Hamiltonian system~\eqref{HamsysOri}.
We remark that in \cite{HMHS}, the authors preformed one further transformation of the system by dividing the right hand side terms by the variable $\phi(f,z)$ as defined in \eqref{HamsysOri}
and then considering the resulting standard Hamiltonian system with a well-known form of prey-predator systems. However, this transformation is not necessary 
for our arguments but would introduce singular 
right hand side terms, for which already the existence of weak solutions to a corresponding PDE model are unclear. 
(It might be possible to define renormalised solutions).

The Hamiltonian, which transforms \eqref{orig} into the  Hamiltonian system~\eqref{HamsysOri}
is given by 
\begin{align}
H(f,z) &:= H_1(f) + H_2(z),\label{H}\\[1mm]
H_1(f)&:= -\sigma \log f - (r-1-\sigma) \log (1-f) \ge 0,\label{H1}\\
H_2(z) &:= - (1-\frac{r}{N}) \log z - (\frac{r}{2}-1) \log (1-z) + R(z) \ge 0,  \label{H2}
\end{align}
where $R(z)$ is defined as a primitive of  $\frac{\partial R}{\partial z}$, which in return is introduced 
by the following definition
\begin{equation}\label{H2z}
 \frac{\partial H_2}{\partial z} = -\frac{(1-\frac{r}{N})}{z} + \frac{(\frac{r}{2}-1)}{1-z} + \frac{\partial R}{\partial z}:= 
-\frac{G(z)}{ z(1-z)(1-z^{N-1})}.
\end{equation}
It can be shown that $\frac{\partial R}{\partial z}$ is a bounded function on $z\in[0,1]$ (see Lemma~\ref{lem:Rzbound} below) and that 
the non-negativity $H_2\ge 0$ follows from choosing a sufficiently large positive integration constant 
in the definition of $R(z)$. 

Before we state further properties of $R$, we note first that 
$$
-\frac{\partial H_1}{\partial f} = \frac{\sigma}{ f} - \frac{r-1-\sigma}{1-f} =  \frac{\sigma - f(r-1) }{f(1-f)}
$$ 
and system \eqref{orig} can be indeed written as a Hamiltonian system of the form \eqref{HamsysOri}, i.e.
\begin{equation*}%
\begin{cases}
\dot{f} = \frac{\partial H_2}{\partial z}\,\phi(f,z),\\[2mm]
\dot{z} = -\frac{\partial H_1}{\partial f}\,\phi(f,z),
\end{cases}
\end{equation*}
with 
$\phi(f,z)=f(1-f)z(1-z)(1-z^{N-1})$ as in \eqref{HamsysOri}.


\begin{lem}\label{lem:G}
The function $G(z)=O(1-z)$ with
$$
\lim_{z\to 1} \frac{-G(z)}{1-z} = \frac{(r-2)(N-1)}{2}.
$$
Hence, Assumption \eqref{para} implies $\lim_{z\to 1} \frac{-G(z)}{1-z}>0$ and, therefore,
$G(z)<0$ for $z$ sufficiently close to 1.
\end{lem}
\begin{proof}
Straightforward polynomial division shows  
\begin{equation}\label{aj}
\frac{-G(z)}{1-z} = \sum_{j=1}^{N-1} \Bigl(r-1-\frac{j\,r}{N}\Bigr) z^{N-1-j} =: \sum_{j=1}^{N-1} a_j\, z^{N-1-j},
\end{equation}
where the coefficients $a_j:=r-1-\frac{j\,r}{N}$ change sign exactly once between $a_1=r-1-r/N>0$ and $a_{N-1}=-1+r/N<0$,
which reflects the single sign change of $G(z)$. 

Alternatively, we set $N-1-j=k$ and define $b_k = -1 + r/N(k+1)$
for $k=0,\ldots,N-2$ to write
\begin{equation}\label{bk}
\frac{-G(z)}{1-z} = \sum_{k=0}^{N-2} \Bigl(-1+\frac{r}{N}(k+1)\Bigr) z^{k} =: \sum_{k=0}^{N-2} b_k\, z^{k}
\end{equation}
with $b_0=-1+r/N<0$ and $b_{N-2}=r-1-r/N>0$. 

Altogether, $G(z)=O(1-z)$ with
$$
\lim_{z\to 1} \frac{-G(z)}{1-z} = \sum_{j=1}^{N-1} a_j =   \sum_{j=0}^{N-2} b_k = \frac{(r-2)(N-1)}{2}>0
$$
for $r>2$ and thus $G(z)<0$ for $z$ sufficiently close to 1.
\end{proof}

\begin{lem}\label{lem:Rzbound}
The rational function $\frac{\partial R}{\partial z}$ is bounded on $z\in[0,1]$.
\end{lem}
\begin{proof}
We calculate from \eqref{H2z} 
\begin{align}
\frac{\partial R}{\partial z}&= 
\frac{-1}{ z(1-z)(1-z^{N-1})}\left[G(z)-\Bigl(1-\frac{r}{N}\Bigr)(1-z)(1-z^{N-1}) 
+ \Bigl(\frac{r}{2}-1\Bigr)z(1-z^{N-1})\right]\nonumber \\
&=\frac{-1}{ (1-z)(1-z^{N-1})}\biggl[\underbrace{\frac{r}{2}\Bigl(1-\frac{4}{N}\Bigr)\! - \frac{r}{N}\sum_{j=1}^{N-3} z^j + r\Bigl(1-\frac{2}{N}\Bigr)z^{N-2} - \frac{r}{2}\Bigl(1-\frac{2}{N}\Bigr) z^{N-1} }_ {=:Q(z)}\biggr]\label{Rone}
\end{align}
Moreover, with $Q(z)$ being defined as the square bracket in \eqref{Rone}, i.e.
\begin{equation}\label{Q}
Q(z):= \frac{r}{2}(1-\frac{4}{N}) - \frac{r}{N}\sum_{j=1}^{N-3} z^j + r(1-\frac{2}{N})z^{N-2} - \frac{r}{2}(1-\frac{2}{N}) z^{N-1}
\end{equation}
we compute that 
\begin{equation}
\lim_{z\to 1} \frac{Q(z)}{(1-z)^2} = \frac{r}{12 N}(N-6)(N-2)(N-1) \qquad \Rightarrow \qquad
Q(z) = (1-z)^2 P(z),
\end{equation}
where $P(z)$ is a polynomial in $z$ of order $N-3$.
Thus, we have 
\begin{equation}\label{Rprime}
\frac{\partial R}{\partial z} = \frac{-P(z)}{1+\sum_{j=1}^{N-2}z^j},
\end{equation}
which is a bounded rational function on $z\in[0,1]$.
\end{proof}
\medskip

\subsection{Positive definiteness of the Hamiltonian}

In the following, we need that the Hessian of the Hamiltonian $H=H_1(f)+H_2(z)$, i.e. 
\begin{equation}\label{HessMatrHam}
D^2 H = 
\begin{pmatrix}
\frac{\partial^2 H_1}{\partial f^2} = \frac{\sigma}{f} + \frac{r-1-\sigma}{(1-f)^2}>0 & 0\\ 
0 & 
\frac{\partial^2 H_2}{\partial z^2} = \frac{(1-\frac{r}{N})}{z^2} + \frac{(\frac{r}{2}-1)}{(1-z)^2} + \frac{\partial^2 R}{\partial z^2}
\end{pmatrix}>0
\end{equation}
is a positive definite matrix. This is obviously true if and only if $\frac{\partial^2 H_2}{\partial z^2}> 0$.
As example, for $N=3$, we calculate especially
\begin{align*}
\frac{\partial R}{\partial z} = 
\frac{-1}{ (1-z)^2(1+z)}\left[-\frac{r}{6} + \frac{r}{3}z - \frac{r}{6} z^{2} \right]
= \frac{r}{6}\frac{1}{1+z}
\qquad\text{and}\qquad\frac{\partial^2 R}{\partial z^2} =-\frac{r}{6}\frac{1}{(1+z)^2}
\end{align*}
and obtain
\begin{equation}\label{N3}
N=3: \quad \frac{\partial^2 H_2}{\partial z^2} 
= \frac{(1-\frac{r}{3})}{z^2} + \frac{(\frac{r}{2}-1)}{(1-z)^2} -\frac{r}{6}\frac{1}{(1+z)^2}
>
\frac{-\frac{r}{3}+\frac{r}{2}}{(1+z)^2} 
-\frac{\frac{r}{6}}{(1+z)^2}=0
\end{equation}
and hence positive definiteness for all $z\in[0,1]$. 
For $N=4$, we obtain also positive definiteness since
\begin{align*}
\frac{\partial R}{\partial z} 
= \frac{r}{4}\frac{z}{1+z+z^2}
\qquad
\text{and} 
\qquad
\frac{\partial^2 R}{\partial z^2} =\frac{r}{4}\frac{1-z^2}{(1+z+z^2)^2}>0,
\end{align*}
which implies
\begin{equation}\label{N4}
N=4:\quad \frac{\partial^2 H_2}{\partial z^2} 
= \frac{(1-\frac{r}{4})}{z^2} + \frac{(\frac{r}{2}-1)}{(1-z)^2} + \frac{\partial^2 R}{\partial z^2}>0, 
\qquad \text{for all}\quad z\in[0,1].
\end{equation}

However, the following Lemma~\ref{HzzPosSemDef} proves not only sufficient conditions for the positive definiteness of the Hessian $D^2 H$, but also that 
$\frac{\partial^2 H_2}{\partial z^2}<0$ is possible. 
\begin{lem}[Positive definiteness of the Hessian of the Hamiltonian]\label{HzzPosSemDef}\hfill\\
Let $r$ satisfy $\max\{\frac{N}{3},2\}<r<N$ for $N\geq3$. Then, 
\begin{equation}\label{H2zz}
\frac{\partial^2 H_2}{\partial z^2}
 >0, \qquad \forall z\in[0,1].
\end{equation}
On the other hand, for $r$ close to 2 and $N$ large, we find $\frac{\partial^2 H_2}{\partial z^2} <0$.
\end{lem}
\begin{proof}
The proof applies different estimates on two intervals 	for $r,$ first $\frac{N}{2}\le r<N$ and secondly $\max\{\frac{N}{3},2\}<r<\frac{N}{2}$. The presentation of the proof will be divided accordingly. 

We will begin our analysis for $\frac{N}{2}\le r<N$. By differentiating \eqref{H2z} with respect to $z$ 
and using the representation \eqref{bk}, 
we derive the following formula for the second derivative of $H_2$: 
\begin{align}\label{H2zz2}
\frac{\partial^2H_2}{\partial z^2}
&=
\frac{S(z)}{z^2(1-z^{N-1})^2}
\end{align}
with 
\be\label{S(z)}
S(z):=
\sum_{k=1}^{N-2}b_k(k-1)z^k+b_0(Nz^{N-1}-1)
+
\sum_{k=1}^{N-2}b_k(N-k)z^{N+k-1}
\ee
and
\begin{equation*}
b_k=-1+\frac{r}{N}(k+1) 
\end{equation*}
as defined in \eqref{bk}.
We notice in the range $\frac{N}{2}\le r<N$ that 
$$
b_0=\frac{r}{N}-1<0,\qquad\text{but}\qquad
b_1,\ldots,b_{N-2}\ge 0.
$$

Hence, in \eqref{S(z)} the only negative term is the middle one, i.e. 
$b_0(Nz^{N-1}-1)$, while the two sums in \eqref{S(z)} contain only non-negative terms for $\frac{N}{2}\le r<N$. Therefore, in order to prove $\frac{\partial^2H_2}{\partial z^2}>0$ in \eqref{H2zz2}, it is sufficient to prove
\be\label{SuffCond}
b_{N-2}(N-3)z^{N-2}+b_0(Nz^{N-1}-1)>0,\quad \text{on}\quad z\in(0,1)
\qquad\text{and}\qquad S(z)=O(1-z),
\ee
where the first term on the above relation is just the last term of the first sum of \eqref{S(z)}. Relation \eqref{SuffCond} with recalling $b_{N-2}$ from \eqref{bk} can be rewritten as 
\begin{multline}\label{bb}
b_{N-2}(N-3)z^{N-2}+b_0(Nz^{N-1}-1)\\
	=\Bigl(r\frac{N-1}{N}-1\Bigr)(N-3)z^{N-2}+\Bigl(\frac{r}{N}-1\Bigr)(N-1)z^{N-1}+\underbrace{\Bigl(\frac{r}{N}-1\Bigr)(z^{N-1}-1)}_{=O(1-z)\ge0}.
\end{multline}	
Since the last term is non-negative, it is sufficient to show 
\begin{equation}\label{bbb}
z^{N-2}\biggl[\Bigl(r\frac{N-1}{N}-1\Bigr)(N-3) - \underbrace{\Bigl(1-\frac{r}{N}\Bigr)(N-1)}_{\ge0} z\biggr] > 0,\qquad \text{on}\quad z\in(0,1).
\end{equation}
By observing that the above bracket is monotone increasing in $r$ for $N\ge3$, we can estimate further below by setting $r=\frac{N}{2}$ and $z=1$ to obtain
\begin{equation*}
\left[\Bigl(\frac{N-1}{2}-1\Bigr)(N-3)-\frac{N-1}{2}\right]=\frac{(N-5)(N-2)}{2}\ge 0, \qquad \text{for}\quad N\ge5.
\end{equation*}
The above estimates proves that the left hand side of \eqref{bbb} is $O(1)$ as $z\to1$ for $N\ge6$ while it is $O(1-z)$ for $N=5$. Consequentially, the term \eqref{bb} is also $O(1-z)$.
Hence, this is sufficient to prove \eqref{SuffCond}
in the considered range $\frac{N}{2}\le r<N$ for $N\ge5$.
Moreover, since positive definiteness for $N=3$ and $N=4$ was already shown in \eqref{N3} and \eqref{N4}, respectively, this completes the proof in the range $\frac{N}{2}\le r<N$.
\medskip

Now, we will treat the second range $\max\{2,\frac{N}{3}\}<r<\frac{N}{2}$, which is more technical. 
Since positive definiteness for $N=3$ and $N=4$ was already shown in \eqref{N3} and \eqref{N4}, we shall furthermore assume $N\ge5$.
In fact, for $N\ge5$, we are able to treat the 
range $\frac{N}{3}<r<\frac{N}{2}$
which implies $\max\{2,\frac{N}{3}\}<r<\frac{N}{2}$.
We consider 
\begin{equation}
	\label{S(z)2}
S(z)=
\underbrace{\sum_{k=2}^{N-2}b_k(k-1)z^k}_{\displaystyle=:S_1(z)}
+
\underbrace{b_0(Nz^{N-1}-1)}_{\displaystyle=:S_2(z)}
+
\underbrace{b_1(N-1)z^N}_{\displaystyle=:S_3(z)}
+
\underbrace{\sum_{k=2}^{N-2}b_k(N-k)z^{N+k-1}}_{\displaystyle=:S_4(z)}.
\end{equation}

Next, we observe that in the range $\frac{N}{3}<r<\frac{N}{2}$ holds 
$$
b_0=-1+\frac{r}{N}<0,
\quad b_1=-1+\frac{2r}{N}<0,
\qquad b_k=-1+\frac{r}{N}(k+1)>0,\quad \forall k\geq2.
$$
therefore
$$
S_1(z)\geq0,\qquad S_3(z)\leq0,\qquad S_4(z)\geq0.
$$
The second term $S_2(z)$ changes sign on $z\in[0,1]$ from a positive constant (at $z=0$) to a negative value (at $z=1$) with a single root.
In order to control $S_2$ and $S_3$ and show \eqref{H2zz2}, it turns out that we  only need to consider the last two terms of both $S_1$ and $S_4$. 
Hence, we shall prove the non-negativity of the following partial sum, denoted by $S_p(z)$, as a sufficient condition 
\begin{equation}\label{range2}
\begin{aligned}
\text{on}\ z\in[0,1): \quad &0<S_p(z,N,r):=
 \underbrace{b_{N-2}(N-3)z^{N-2}}_{\displaystyle=:S_{1a}}
+\underbrace{b_{N-3}(N-4)z^{N-3}}_{\displaystyle=:S_{1b}}\\
&\qquad\qquad\qquad\quad\ \ +S_2+S_3 +\underbrace{2b_{N-2}z^{2N-3}}_{\displaystyle=:S_{4a}}
+\underbrace{3b_{N-3}z^{2(N-2)}}_{\displaystyle=:S_{4b}}. \\
\text{as}\ z\to\ 1: \quad &0 \le S_p(z,N,r) = O((1-z)^2) \quad \text{with}\quad \lim_{z\to 1-} \frac{O((1-z)^2)}{(1-z)^2}>0. 
\end{aligned}
\end{equation}
From these six expressions only $S_2$ and $S_3$ can be negative. More precisely,
$$
S_{1a},S_{1b},S_{4a},S_{4b}\geq0,\qquad S_3\leq0, 
\qquad 
S_2(z) =\begin{cases} >0 & z\in[0,z^*), \\
<0 & z\in(z^*,1], 
\end{cases}
\quad z^*=\left(\frac{1}{N}\right)^\frac{1}{N-1}.
$$ 

Accordingly, in order to prove \eqref{range2}, 
we split the interval $z\in[0,1]$ into the intervals $\mathrm{I}=[0,z^*]$, where instead for proving \eqref{range2}, it will be sufficient to show $S_{1a}+S_{1b}+S_3>0$ on $z\in(0,z^*]$
and the interval $\mathrm{II}=[z^*,1]$, where we requires all six terms of $S_p$ to prove the sufficient condition \eqref{range2}.

\underline{We begin with the first interval $\mathrm{I}=[0,z^*]$.}
Since $S_2\ge0$ on $\mathrm{I}$, it is sufficient to show $S_{1a}+S_{1b}+S_3>0$ on $z\in(0,z^*]$, i.e. 
\begin{equation}\label{NNN}
\left[r\Bigl(1-\frac1N\Bigr)-1
\right](N-3)z^{N-2}
+\left[
r\Bigl(1-\frac2N\Bigr)-1
\right](N-4)z^{N-3}
-\Bigl(1-\frac{2r}{N}\Bigr)(N-1)z^n>0.
\end{equation}
First, we observe that these three terms are all monotone increasing in $r$ for $N\ge4$. With $r>\max\{\frac{N}{3},2\}\ge2$, 
it is thus sufficient to set $r=2$ in \eqref{NNN} in order to prove \eqref{range2} 
on the interval $\mathrm{I}$. Hence, after multiplying with $N$, we obtain the sufficient condition 
$$
(N-2)(N-3)z^{N-2}+(N-4)^2 z^{N-3}-(N-4)(N-1)z^N>0, 
$$
Furthermore, since $z^{N-2},z^{N-3}\ge z^{N}$, it sufficient to show 
$$
(N-4)^2+(N-2)(N-3)-(N-4)(N-1)\ge (N-4)^2+2>0,
$$
which proves \eqref{NNN}. Finally, since $S_p(0,N,r)=-b_0>0$, we have \eqref{range2} in the interval $\mathrm{I}$.

\underline{We continue with the second interval $\mathrm{II}=[z^*,1]$}. Here, since $Nz^{N-1}-1\ge0$ and $S_2\le0$,
the fact that the coefficients $b_k$ in \eqref{S(z)2}
are monotone increasing in $r$ implies that all six terms in  
$S_p(z,N,r)$ are monotone increasing in $r$. 
Thus, it is sufficient to prove $S_p(z,N,2)>0$:
\begin{align*}
	S_p(z,N,2)&=\left[2\Bigl(1-\frac1N\Bigr)-1\right](N-3)z^{N-2}+\left[2\Bigl(1-\frac2N\Bigr)-1\right](N-4)z^{N-3}
	\nonumber\\
&\quad 	-\Bigl[1-\frac2N\Bigr](Nz^{N-1}-1)
	-\Bigl[1-\frac4N\Bigr](N-1) z^{N} 
	\nonumber\\
	&\quad +2\left[2\Bigl(1-\frac1N\Bigr)-1\right]z^{2N-3}
	+3\left[2\Bigl(1-\frac2N\Bigr)-1\right]z^{2N-4}.\nonumber
\end{align*}
By collecting the term proportional to 
$\bigl[1-\frac2N\bigr]$ and $\bigl[1-\frac4N\bigr]$, respectively, we obtain
\begin{align*}
S_p(z,N,2)&=\Bigl[1-\frac4N\Bigr] 
\underbrace{\left[(N-4)z^{N-3}-(N-1) z^{N}+3z^{2N-4}\right]}_{=z^{N-3}\left[(N-1)(1-z^3)-3(1-z^{N-1})\right]=:z^{N-3}R_1}
\\
&\quad+\Bigl[1-\frac2N\Bigr] \underbrace{\left[
(N-3)z^{N-2}-(Nz^{N-1}-1)
+2z^{2N-3}\right]}_{=N z^{N-2}(1-z)-2z^{N-2}(1-z^{N-1})+1-z^{N-2}=:R_2}.
\end{align*}
Hence, for $N\ge5$, 
we estimate first
\begin{align*}
R_1&=(1-z)\Bigl[(N-1)(1+z+z^2)-3(1+z+z^2+\underbrace{z^3+\ldots + z^{N-2}}_{\le N-4})\Bigr]\\
&\ge (1-z)\left[(N-4)(1+z+z^2)-3 z^3 (N-4)\right]\\
&=(1-z)(N-4)\left[1+z+z^2-3 z^3\right]
=(1-z)^2(N-4)(1+2z+3z^2)\ge0,
\end{align*}
which implies the strict inequality $R_1>0$ for all $z\in[z^*,1)$ and $\lim_{z\to1-} \frac{R_1}{(1-z)^2}>0$.
Secondly, we calculate
\begin{align*}
R_2&= (1-z)\left[N z^{N-2}-2z^{N-2}(1+\ldots+z^{N-2})+(1+\ldots+z^{N-3})\right]\\
&= (1-z)\left[(N-2) z^{N-2}-2z^{N-1}(1+\ldots+z^{N-3})+(1+\ldots+z^{N-3})\right]\\
&=(1-z)\left[(N-2) z^{N-2}+(1+\ldots+z^{N-3})(1-z^{N-1})-z^{N-1}(1+\ldots+z^{N-3})\right]\\
&=(1-z)\bigl[(1+\ldots+z^{N-3})(1-z^{N-1})+z^{N-2}\bigl[(N-2)-\underbrace{(z+\ldots+z^{N-2})}_{\le N-2}\bigr]\bigr]\ge0,
\end{align*}
which implies the strict inequality $R_2>0$ for all $z\in[z^*,1)$ and $\lim_{z\to1-} \frac{R_2}{(1-z)^2}>0$.
This proves \eqref{range2} on the Interval 
$\mathrm{II}$.

Finally, Figure~\ref{counter} illustrates in the limiting case $r=2$ the sign of 
$\frac{\partial^2H_2}{\partial z^2}$ (by plotting 
$\frac{\partial^2H_2}{\partial z^2}\times N(z-1)^2 z^2 (z-z^N)^2$ in order to avoid plotting singularities at $z=0,1$).
We observe that for $16\lesssim N$, there are values near $z\sim 0.7$ where $\frac{\partial^2H_2}{\partial z^2}<0$.
\begin{figure}\begin{center}
\includegraphics{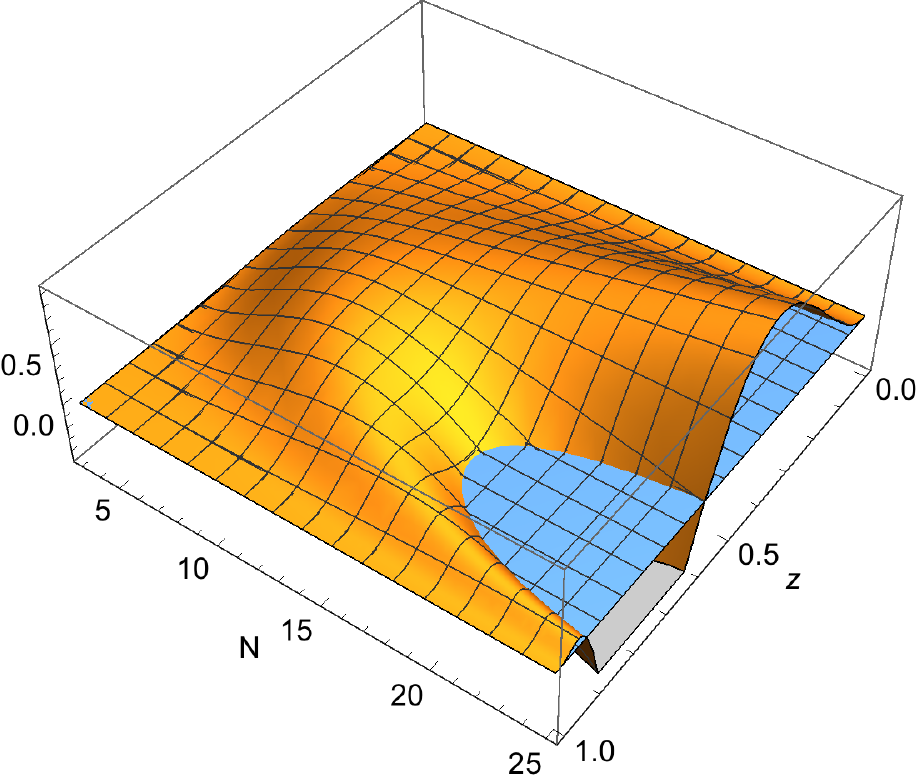}
\caption{\label{counter}The sign of $\frac{\partial^2H_2}{\partial z^2}$ in the limiting case $r=2$ as function of 
$N$ and $z$. For $16\lesssim N$, there is a blue region near $z\sim0.7$ where $\frac{\partial^2H_2}{\partial z^2}<0$.}
\end{center}
\end{figure}
\end{proof}

\section{Proof of Theorem \ref{thm:1}}
We begin by a brief outline the proof of Theorem \ref{thm:1}: First, we reformulate \eqref{HamPDEsys0}
in terms of the Hamiltonian structure of ODE model 
\eqref{HamsysOri}. Then, we notice that the spatially integrated ODE Hamiltonian constitutes also 
a Ljapunov functional for the PDE system \eqref{HamPDEsys0}, i.e. 
\begin{equation}\label{PDEHam}
\mathcal{H}(f,z)(t):=\int_\Omega H(f(x,t),z(x,t))\,dx
\qquad \Rightarrow \qquad \frac{d}{dt}\mathcal{H}\le 0
\end{equation} 
Next, we prove that the $\omega$-limit set is non-empty, compact and connected.
Afterwards, for trajectories $\tilde{w}$ of the $\omega$-limit set, we show that $\int_\Omega H(\tilde{w})\;dx$ is  well defined.
and we get that the $\omega$-limit set is invariant. Finally, we notice that asymptotically, we have a spatially homogeneous and periodic in time orbit.

\medskip
\noindent{\bf Hamiltonian structure and global classical solutions:}
%
%
%
Recalling the ODE model \eqref{HamsysOri}, we rewrite 
the PDE system \eqref{HamPDEsys0} as
\begin{equation}\label{New_HamPDEsys0}
\begin{cases}
f_t-d_f\Delta f= \phi(f,z) \,H_z,\\
z_t-d_z\Delta z=-\phi(f,z) \,H_f,\\
\frac{\partial}{\partial\nu}(f,z)|_{\partial\Omega}=0,
\end{cases}
\end{equation}
where we denote
$$
H_f = \frac{\partial H}{\partial f}= \frac{\partial H_1}{\partial f},\qquad\text{and}\qquad
H_z = \frac{\partial H}{\partial z}=\frac{\partial H_2}{\partial z}.
$$
Straightforward calculation yields  
\begin{align}
\frac{d}{dt} \mathcal{H}(f,z) &= \frac{d}{dt}\int_\Omega H(f,z)\;dx=\int_\Omega (H_ff_t+H_zz_t)\;dx	
\nonumber\\
&=\int_\Omega (H_f \phi H_z+d_fH_f\Delta f- H_z\phi H_f+d_zH_z\Delta z)\;dx
\nonumber\\
&=-\int_\Omega(d_fH_{ff}|\nabla f|^2+d_zH_{zz}|\nabla z|^2)\;dx<0, \qquad \text{for all } \ (\nabla f, \nabla z)\neq 0.
\label{HaMon}
\end{align}
where \eqref{HaMon} holds provided that $H_{zz}>0$  
(see e.g. Lemma \ref{HzzPosSemDef}) because 
$$
H_{ff}=\frac{\sigma}{f} + \frac{r-1-\sigma}{(1-f)^2}>0,\qquad\text{due to}\quad \sigma<r-1.
$$
Hence, $\mathcal{H}(f,z)$ constitutes a Ljapunov functional for the PDE system \eqref{HamPDEsys0}. Moreover, the Ljapunov functional $\mathcal{H}(f,z)\ge0$ is non-negative since $H_1\ge0$ and $H_2\ge0$.

Next, we remark that global-in-time solutions 
to the PDE system \eqref{HamPDEsys0} follow from 
standard parabolic theory (e.g. invariant regions, see \cite{SmollerBook}) or weak comparison principle arguments ensuring $0\le f(x,t),z(x,t) \le 1$ a.e. $x\in\Omega$ for all $t\ge 0$ 
provided that  $0\le f(x,0),z(x,0) \le 1$.
Moreover, standard parabolic regularity implies classical $C^2$ solutions due to the assumed regularity of $\partial\Omega$.
\medskip

\noindent{\bf Orbits and $\omega$-limit set:}
We define the solutions' orbits of the PDE system \eqref{HamPDEsys0}:
$$
\mathcal{O}
=\left\{w(\cdot,t)\right\}_{t\geq0}=
\left\{(f(\cdot,t),z(\cdot,t))\right\}_{t\geq0}\subset
(C^2(\overline{\Omega}))^2,
$$ 
which is compact and connected in $(C^2(\overline{\Omega}))^2$.

Define
$
w=(f,z),\ w=w(\cdot,t)\in\mathbb{R}^2,\ w|_{t=0}=w_0\geq0
$
and the $\omega-$limit set:
$$
\omega(w_0)=\left\{
w_*=(f_*,z_*)\ |\ \exists t_k\uparrow+\infty:\|(f(\cdot,t_k),z(\cdot,t_k))-(f_*,z_*)\|_{C^2(\Omega)}=0
\right\}.
$$
and $w_*:=(f_*,z_*)$ is the limit of $w(\cdot,t_k)=(f(\cdot,t_k),z(\cdot,t_k))$ as $t_k\to\infty$.

Since the Hamiltonian is monotone decreasing, see \eqref{HaMon}, from general parabolic theory we get that 
$\emptyset\neq\omega(w_0)\subset(C^2(\overline{\Omega}))^2 $ is compact and  connected, therefore a semi-flow (since it is defined only for non-negative values of $t$) is well defined on $\omega(w_0)$. Moreover, $\omega(w_0)$ is invariant under this flow: Taking an element from the $\omega-$limit set, $\tilde{w_0}\in\omega(w_0)$, we define a solution $\tilde{w}=\tilde{w}(\cdot,t)$ for $t\ge0$ with $\tilde{w}|_{t=0}=\tilde{w_0}$. This solution  $\tilde{w}$ also exists globally in time and  $\tilde{w}(\cdot,t)\in \omega(w_0)$.

{At this point we can see that $\mathcal{H}(w)=\int_\Omega H(w)dx$ is not necessarily well-defined for $w=w_\infty\in\omega(w_0)$ since we only know $0\leq w_\infty\leq1$. In order to exclude that possibility, we consider 
for any $w_\infty \in \omega(w_0)$
a solution ${w}={w}(\cdot,t)$ to the PDE system \eqref{HamPDEsys0} for $t\ge 0$ with $\lim_{t\to\infty}{w}=w_\infty$. 
Since we assumed initial data with $\mathcal{H}(w(\cdot,0))\leq C$, by using Fatou's Lemma, the following arguments shows $w_\infty\not\equiv0,1$: First, since}
$$
w(\cdot,t_k)\to w_\infty\quad\text{uniformly on}\ C^2(\overline{\Omega})\times C^2(\overline{\Omega})
$$
and
$$
\int_\Omega H(w(x,t_k))\;dx<+\infty
$$
we get 
$$
\int_\Omega H(w_\infty)\leq \liminf_{k\to\infty}\int_\Omega H(w(x,t_k))dx<+\infty.
$$
With $H_1$ and $H_2$ being non-negative, see \eqref{H1},\eqref{H2}, this implies 
$$
\int_\Omega H_1(f_\infty(x))\;dx, \int_\Omega H_2(z_\infty(x))\;dx<+\infty
$$
The first bound writes as 
$$
 \int_\Omega (-\sigma \log f_\infty - (r-1-\sigma) \log (1-f_\infty))\;dx<+\infty
$$
and since both terms $-\sigma \log f_\infty$ and $(r-1-\sigma) \log (1-f_\infty)$ are non-negative, we get $$
 -\int_\Omega  \log f_\infty\;dx<+\infty \quad\Rightarrow\quad f_\infty\not\equiv0
$$
and from
$$
- \int_\Omega \log (1-f_\infty)\;dx<+\infty  \quad\Rightarrow\quad f_\infty\not\equiv1.
$$
similarly we derive that $z_\infty\not\equiv0,1$ and thus conclude that $w_\infty\not\equiv0,1$.

{Returning to the above solution $\tilde{w}=\tilde{w}(\cdot,t)$ subject to $\tilde{w}|_{t=0}=\tilde{w_0}\in\omega(w_0)$,
it follows that at some points $x\in \Omega$, the initial data  
 $\tilde{w_0}=(f_\infty,z_\infty)$ satisfy $f_\infty(x), z_\infty(x) >0$ and $f_\infty(x), z_\infty(x) <1$.}
Thus, the strong maximum principle 
implies that for all $t>0$ holds $0<\tilde{w}(\cdot,t)<1$ and thus $\mathcal{H}(\tilde{w})$ is a well defined for $t>0$. 

Next, La-Salle's principle implies that $\mathcal{H}(\tilde{w}(\cdot,t))$ is invariant and
$$
\frac{d}{dt}\mathcal{H}(\tilde{w}(\cdot,t))=0,\quad t>0,
$$
Thus, $\tilde{w}$ is spatially homogeneous, i.e. $\nabla \tilde{w}(\cdot,t)=0$ by \eqref{HaMon} 
and parabolic smoothness, and by letting $t\downarrow0$ we get that $\nabla w_\infty=0$. Therefore, $\omega(w_0)\subset \{w_\infty\in\mathbb{R}^2|0<w_\infty<1\}$.

{Next, we proceed as in \cite{KarSuzYam} since we have an asymptotically spatial homogeneous and periodic-in-time orbit (recall that spatial homogeneous solutions $\tilde{w}$ are solutions to the ODE-model~\ref{orig} and hence periodic, see  \cite{HMHS}) } and then from \eqref{HaMon} and parabolic regularity we get that 
\begin{equation}\label{NablaC1}
\lim_{t\uparrow\infty}
\|\nabla w(\cdot,t)\|_{C^1}
=0
\end{equation}
and thus
the convergence \eqref{eqn:convergence}, i.e. 
$$
\lim_{t\uparrow+\infty}\mbox{dist}_{C^2}(w, {\cal O})=0.
$$

\section{Proof of Corollary \ref{Cor}}
The proof of Corollary \ref{Cor} is based on the ideas of the second part of Theorem 1.1 in \cite{KarSuzYam}, 
which we outline for the convenience of the reader. 
First, we prove the following
\medskip

\noindent{\bf Claim 1:} 
Under the assumptions of Theorem \ref{thm:1}, each sequence $t_k\uparrow\infty$ admits {a subsequence} $\{t_k'\}\subset\{t_k\}$ and a solution $(\tilde f,\tilde z)$ of the ODE model \eqref{orig}, such that ${\cal O}=\{ (\tilde f(t), \tilde z(t))\}_{t\in\mathbb{R}}$ and 
\begin{equation}\label{conv0}
\lim_{k\to+\infty}\sup_{t\in[-T,T]}\|
(f(\cdot,t+t_k'),z(\cdot,t+t_k'))
-
(\tilde f(t),\tilde z(t))
\|_{C^2}=0	
\end{equation}
for any $T>0$.

We begin by observing that the previous Theorem \ref{thm:1} and parabolic 
regularity implies {for any solution 
$(f(\cdot,t),z(\cdot,t))$ to \eqref{HamPDEsys0}}
the existance of a constant $C>$ such that 
for positive times, for instance, for $t\geq1$ holds
$$
\|f_t(\cdot,t)\|_{C^2}+\|z_t(\cdot,t)\|_{C^2}\leq C,\qquad\forall t\geq1.
$$
Then, by the theorem of Ascoli-Arzel\'a, the sequence $\{t_k\}\uparrow\infty $ admits a subsequence $\{t_k'\}\subset\{t_k\}$ and a solution $(\hat f(\cdot,t),\hat z(\cdot,t))$ to \eqref{HamPDEsys0} such that for any $T>0$ (cf. \cite[Lemma 3.6]{KarSuzYam})
\begin{equation}\label{conv1}
\lim_{k\to+\infty}\sup_{t\in[-T,T]}\|
(f(\cdot,t+t_k'),z(\cdot,t+t_k'))
-
(\hat f(\cdot,t),\hat z(\cdot,t))
\|_{C^2}=0	.
\end{equation}
From relation \eqref{NablaC1}, we derive that
$$
\nabla\hat f(\cdot,t)
=
\nabla\hat z(\cdot,t)
=0,\qquad
\forall t\in[-T,T].
$$
Hence, the solution $(\hat f,\hat z)$ must be spatially homogeneous and, thus, it is a solution 
to the ODE model \eqref{orig} and we 
denote it by $(\tilde f(t),\tilde z(t))$ in the following. Moreover, \eqref{conv0} follows from \eqref{conv1}.
\medskip

\noindent{\bf Claim 2:}  Now that we have proven the Claim 1, we can show \eqref{l}.

Denote by $l\geq0$ to be the time period of the above ODE solution 
$(\tilde f(t),\tilde z(t))$ to \eqref{orig} on $\cal O$ as given in Claim 1. 
{Recall that \eqref{orig} was shown in \cite{HMHS} to have closed orbits around a single fix point in the parameter range \eqref{para}.
If the considered ODE-orbit $\cal O$ does not only contain the fix point, we have $l>0$.} We then take $T>2l$. 

By the previous claim, any sequence $\{t_k\uparrow\infty\}$ admits a subsequence $\{t_k'\}\subset\{t_k\}$ and a solution $(\tilde f(t),\tilde z(t))$ to \eqref{orig} such that ${\cal O}=\{ (\tilde f(t), \tilde z(t))\}_{t\in\mathbb{R}}$ and 
\eqref{conv0} holds true. Let a fixed $t\in[-T,T]$, then the periodicity
\begin{equation}\label{period}
	(\tilde f(t+l),\tilde z(t+l))
=
(\tilde f(t),\tilde z(t))
\end{equation}
allows to calculate
\begin{align*}
	\limsup_{k\to\infty}\|
	&\left(f(\cdot,t+l+t_k'),z(\cdot,t+l+t_k')\right)
	-
		\left(f(\cdot,t+t_k'),z(\cdot,t+t_k')\right)
	\|_{C^2}
	\\&\leq
	\lim_{k\to\infty}\|
	\left(f(\cdot,t+l+t_k'),z(\cdot,t+l+t_k')\right)
	-
		(\tilde f(t+l),\tilde z(t+l))
	\|_{C^2}
	\\&\quad+
	\lim_{k\to\infty}\|
		\left(f(\cdot,t+t_k'),z(\cdot,t+t_k')\right)
		-
		(\tilde f(t),\tilde z(t))
	\|_{C^2}=0,
\end{align*}
where we have added and subtracted relation \eqref{period} and used the triangle inequality. Thus, we get
$$
	\lim_{s\to\infty}\|
		\left(f(\cdot,t+l+s),z(\cdot,t+l+s)\right)
		-
		( f(\cdot,t+s), z(\cdot,t+s))\|_{C^2}=0
$$
	and \eqref{l} follows.
	

\section{Proof of Theorem \ref{thm:2}}

At this section, we denote by $C$ various constants that may change from line to line.

We begin by testing the second equation in \eqref{HamPDEsys0} with $z$ and obtain after integration by parts
	\begin{align*}
		\frac12\frac{d}{dt} \|z\|_2^2&= - d_z \|\nabla z\|_2^2    +\int_\Omega z(\sigma - f (r-1)) z(1-z)(1-z^{N-1})
		\\&\leq
		- d_z \|\nabla z\|_2^2  
		+\sigma\int_\Omega z^2(1-z)(1-z^{N-1})\leq
				- d_z \|\nabla z\|_2^2  		+ C \| z\|_2^2  
			\end{align*}
where $C=\sigma>0$ does not depend on $d_z$. Consequentially, by multiplying with $2e^{-2Ct}$, we get
$$
\frac{d}{dt}(e^{-2Ct}\|z\|_2^2)
\leq
-2d_ze^{-2Ct}\|\nabla z\|^2_2
$$
and then
\begin{equation} 
\int_0^T\Vert \nabla z(\cdot,t)\Vert_2^2dt\leq
 \frac{e^{2C{T}}}{2d_z}
 \Vert z_0\Vert_2^2.  
 \label{eqn:d}
\end{equation} 
with a constant $C$ independent of $d_z$.

Next, we apply semi-group estimates for the Laplace operator subject to homogeneous Neumann boundary conditions, see e.g. \cite{QS,Win10}
\begin{equation} 
\Vert \nabla e^{t\Delta}\phi\Vert_s\leq C(q,s)e^{-\mu t}\max\{1, t^{-\frac{d}{2}(\frac{1}{q}-\frac{1}{s})-\frac{1}{2}}\}\Vert \phi\Vert_q, \qquad 1\leq q\leq s\leq \infty,\quad t>0 
 \label{eqn:yuko2}
\end{equation} 
for $0<\mu<\mu_2$, where $\mu_2$ denotes the second eigenvalue of $-\Delta$ (with the Neumann boundary conditions).  Actually, the first equation in \eqref{HamPDEsys0} implies
\begin{equation} 
f(\cdot,t)=e^{td_f\Delta}u_0+\int_0^te^{(t-s)d_f\Delta}\,\varphi_1(f(\cdot,s), z(\cdot,s))ds, 
 \label{yuko1}
\end{equation}  
where
$$
\varphi_1(f,z)=-f(1-f)G(z).
$$
Hence, by taking the gradient of \eqref{yuko1}, it follows from \eqref{eqn:yuko2} that for $q=\infty=s$,
\begin{equation} 
\Vert \nabla f(\cdot,t)\Vert_\infty\leq C, \qquad t \geq 0.  
 \label{eqn:morrey2}
\end{equation} 
Note that again the constant $C>0$ is independent of $d_z>1$.  Similar, we consider 
\begin{equation} 
z(\cdot,t)=e^{td_z\Delta}z_0+\int_0^te^{(t-s)d_z\Delta}\,\varphi_2(f(\cdot,s), z(\cdot,s))ds, 
 \label{eqn:int-2}
\end{equation} 
where
$$
\varphi_2(f,z)=(\sigma - f (r-1)) z(1-z)(1-z^{N-1})
$$
and hence 
\begin{equation} 
\Vert \nabla z(\cdot,t)\Vert_\infty\leq C, \qquad t\geq 0.   
 \label{eqn:morrey3}
\end{equation} 
Again this $C>0$ is independent of $d_z\geq 1$.

Then, since we have already proven $\|f\|_\infty,\|z\|_\infty<1$ in Theorem~\ref{thm:1}, we use (\ref{eqn:morrey2}), (\ref{eqn:morrey3}), and apply (\ref{eqn:yuko2}) with $q=\infty$ to a differentiated version of 
(\ref{yuko1}) ({that is, to the Duhamel formula of differentiated versions of the equation for $f$ in \eqref{HamPDEsys0}})
to get 
\[ 
\Vert \nabla^2f(\cdot,t)\Vert_s\leq C, \qquad t\geq 0, \quad s>d 
\]  
under the assumption $f_0, z_0\in W^{3,s}(\Omega)$ (actually what is needed here is $(f_0, z_0)\in W^{1,\infty}\cap W^{2,s}$),  with $C>0$ independent of $d_f\geq 1$.  
 Similarly, we obtain 
\[ 
\Vert \nabla^2 z(\cdot,t)\Vert_s\leq C, \qquad t\geq 0 
\] 
and hence 
\[ 
\Vert f(\cdot,t)\Vert_{W^{3,s}}+\Vert z(\cdot,t)\Vert_{W^{3,s}}\leq C, \qquad t\geq 0, 
\] 
which implies also 
\[ \Vert f_t(\cdot,t)\Vert_s\leq C, \qquad t\geq 0. \] 

Next, the corresponding family $\{ (f,z)=(f_{d_z}(\cdot,t), z_{d_z}(\cdot, t))\}$ for $d_z\geq 1$ is compact in $C([0, T], C^2(\overline{\Omega})\times C^2(\overline{\Omega}))$ by Morrey's and Ascoli-Arzel\'a's theorems.  Thus, any sequence $(d_z)_k\uparrow+\infty$ admits a subsequence $\{ (d_z)_k'\}\subset\{ (d_z)_k\}$ and $(F,Z)=(F(\cdot,t), Z(\cdot,t))$ such that $(f_{(d_z)_k'}, z_{(d_z)_k'})\rightarrow (F,Z)$ in $C([0,T], C^2(\overline{\Omega})\times C^2(\overline{\Omega}))$.  

The above $F=F(\cdot,t)$ satisfies 
\begin{equation*}
\begin{cases} 
F_t=d_f\Delta F+\varphi_1(F,Z) &\quad \mbox{in $\Omega\times(0,T)$} \\ 
\frac{\partial F}{\partial \nu}=0 &\quad \mbox{on $\partial\Omega\times(0,T)$} 
\end{cases}
\end{equation*} 
subject to $\left. F\right\vert_{t=0}=f_0(x)$. On the other hand, $Z=Z(\cdot, t)$ is independent of $x$ by taking the limit $(d_z)_k\uparrow+\infty$ in (\ref{eqn:d}).  Since $(f,z)=(f_{d_z}(\cdot,t), z_{d_z}(\cdot,t))$ satisfies 
\[ \frac{d}{dt}\aint_\Omega zdx
=\aint_\Omega \varphi_2(f,z)dx, \qquad \left.\aint_\Omega z\ dx\right\vert_{t=0}=\overline{z}_0, \] 
it holds that 
\[ \frac{dZ}{dt}=\aint_\Omega \varphi_2(F,Z)dx \quad \quad \mbox{in $(0,T)$} \]  
subject to $\left. Z\right\vert_{t=0}=\overline{z}_0$.  

Finally, we remark that existence of a unique local-in-time solution to the shadow system (\ref{ssHamPDEsys0}) with (\ref{eqn:shadow2}) follows from standard argument. 
Moreover, from the compactness of the set $\{ (f,z)=(f_{d_z}(\cdot,t), z_{d_z}(\cdot, t))\}$ and the uniqueness of the limit, we conclude the convergence (\ref{eqn:conv}), i.e.
$$
\lim_{d_z\uparrow +\infty}\sup_{t\in [0,T]}\left\{ \Vert f(\cdot, t)-F(\cdot, t)\Vert_{C^2} +\Vert z(\cdot,t)-Z(t)\Vert_{C^2}\right\}=0.  
$$

\section{Proof of Theorem \ref{th3}}
\begin{proof}[Proof of Theorem \ref{th3}]

The shadow system \eqref{ssHamPDEsys0} takes the form 
\begin{align}\label{ssHamPDEsys1}
\begin{cases}
\partial_t F - d_F \Delta F =  -F(1-F)G(Z),
\quad &\text{in}\ \Omega\times(0,T),\\[1mm]
\frac{d}{dt}Z = Z(1-Z)(1-Z^{N-1}) \aint_\Omega(\sigma - F (r-1)) dx,
\quad &\text{in}\ \Omega\times(0,T),\\[1mm]
\frac{\partial F}{\partial \nu}=0,\quad &\text{on}\ \partial\Omega\times(0,T)
,\\[1mm]
F\vert_{t=0}=f_0(x), \qquad  Z\vert_{t=0}=\overline{z}_0,\quad &\text{in}\ \Omega.
\end{cases}
\end{align}
First, we verify that the Hamiltonian \eqref{H}--\eqref{H2} 
also applies to the shadow system \eqref{ssHamPDEsys1}:
\begin{align*}
\frac{d}{dt} \mathcal{H}(F,Z) &=
\int_\Omega (H_f(F) 
F_t+H_{z}(F) Z_t)\,dx	
\nonumber\\
&=\int_\Omega \left[(H_f(F) \phi(F,Z) H_z(Z)+d_fH_f\Delta f-H_z(Z) \aint_\Omega \phi(F,Z) H_f(F)\;dx\right]dx
\nonumber\\
&=-\int_\Omega(d_fH_{ff}|\nabla f|^2)\,dx
<0, \qquad \text{for all } \ (\nabla f, \nabla z)\neq 0,
\end{align*}
since $Z=Z(t)$ and $H_z(Z)$ are spatially homogeneous.

Given the same Hamiltonian as for the PDE model 
\eqref{HamPDEsys0}, we can follow all the steps of the proof of Theorem~\ref{thm:1} to prove Theorem~\ref{th3}.
In particular, global existence of solutions to the shadow system \eqref{ssHamPDEsys1} and the characterisation of the $\omega$-limit set can be performed in the same way. 
\end{proof}

\section*{Acknowledgements}
K.F. acknowledges the kind hospitality of the 
universities of Osaka and Mannheim and was partially supported by NAWI Graz. 
The second author was supported by DFG Project CH 955/3-1. 
Part of the current work was inspired during visits of the first and the  second author at the Department of System Innovation of Osaka University. 
E.L. would like to express his gratitude for the warm hospitality. The third author was partially supported by JSPS Grand-in-Aid for Scientific Research 26247013, 15KT0016, 16H06576 and JSPS core-to-core program Advanced Research Networks.


\end{document}